\documentclass[a4paper]{amsart}
\usepackage{amsmath,amsthm,amssymb,hyperref,mathrsfs}
\usepackage{aliascnt}
\usepackage{lmodern}
\usepackage[T1]{fontenc}

\usepackage{dsfont}

\usepackage[textsize=footnotesize]{todonotes}

\usepackage{xcolor}
\definecolor{dblue}{rgb}{0,0,0.70}
\hypersetup{
	unicode=true,
	colorlinks=true,
	citecolor=dblue,
	linkcolor=dblue,
	anchorcolor=dblue
}

\makeatletter
\expandafter\g@addto@macro\csname th@plain\endcsname{%
		\thm@notefont{\bfseries}
	}%
\expandafter\g@addto@macro\csname th@remark\endcsname{%
		\thm@headfont{\bfseries}
	}%
\makeatother

\newtheorem{theorem}{Theorem}[section]
\newtheorem*{theorem*}{Theorem}

\newaliascnt{lemma}{theorem}
\newtheorem{lemma}[lemma]{Lemma}
\aliascntresetthe{lemma}
\newtheorem*{lemma*}{Lemma}

\newaliascnt{claim}{theorem}
\newtheorem{claim}[claim]{Claim}
\aliascntresetthe{claim}

\newaliascnt{proposition}{theorem}
\newtheorem{proposition}[proposition]{Proposition}
\aliascntresetthe{proposition}

\newaliascnt{corollary}{theorem}
\newtheorem{corollary}[corollary]{Corollary}
\aliascntresetthe{corollary}

\newtheorem*{subclaim}{Subclaim}

\theoremstyle{remark}

\newaliascnt{remark}{theorem}
\newtheorem{remark}[remark]{Remark}
\aliascntresetthe{remark}
\newaliascnt{question}{theorem}
\newtheorem{question}[question]{Question}
\aliascntresetthe{question}

\newtheorem*{question*}{Question}

\newaliascnt{definition}{theorem}
\newtheorem{definition}[definition]{Definition}
\aliascntresetthe{definition}

\newaliascnt{example}{theorem}

\aliascntresetthe{example}

\renewcommand{\restriction}{\mathbin\upharpoonright}

\newcommand{\axiomft}[1]{\mathsf{#1}}
\newcommand{\ZFC}{\axiomft{ZFC}}

\newcommand{\CH}{\axiomft{CH}}

\newcommand{\ZF}{\axiomft{ZF}}

\newcommand{\HOD}{\mathrm{HOD}}

\newcommand{\HS}{\axiomft{HS}}

\DeclareMathOperator{\dom}{dom}
\DeclareMathOperator{\rng}{rng}
\DeclareMathOperator{\supp}{supp}
\DeclareMathOperator{\rank}{rank}
\DeclareMathOperator{\sym}{sym}

\DeclareMathOperator{\fix}{fix}

\DeclareMathOperator{\id}{id}
\DeclareMathOperator{\aut}{Aut}
\DeclareMathOperator{\Col}{Col}
\DeclareMathOperator{\iCol}{Col_{inj}}
\DeclareMathOperator{\Add}{Add}

\newcommand{\forces}{\mathrel{\Vdash}}
\newcommand{\nforces}{\mathrel{\not{\forces}}}

\newcommand\PP{\mathbb{P}}

\newcommand{\sF}{\mathscr F}

\newcommand{\sG}{\mathscr G}

\newcommand{\1}{\mathds{1}}

\newcommand{\tup}[1]{\langle#1\rangle}

\usepackage{enumitem}

\newenvironment{enumerate-(a)}{\begin{enumerate}[label={\upshape (\alph*)}, leftmargin=2pc]}{\end{enumerate}}
\newenvironment{enumerate-(1)}{\begin{enumerate}[label={\upshape (\arabic*)}, leftmargin=2pc]}{\end{enumerate}}

\author{Asaf Karagila}
\address[Asaf Karagila]{School of Mathematics,
University of East Anglia.
Norwich, NR4~7TJ, UK
}
\email[Asaf Karagila]{karagila@math.huji.ac.il}
\author{Philipp Schlicht}
\address[Philipp Schlicht]{School of Mathematics,
University of Bristol,
Fry Building.
Woodland Road,
Bristol, BS8~1UG, UK}
\email[Philipp Schlicht]{philipp.schlicht@bristol.ac.uk}
\date{June 11, 2020}

\title[How to have more things by forgetting how to count them]{How to have more things by\\ forgetting how to count them}

\dedicatory{In loving memory of Matti Rubin}

\begin{document}
\begin{abstract}
Cohen's first model is a model of Zermelo--Fraenkel set theory in which there is a Dedekind-finite set of real numbers, and it is perhaps the most famous model where the Axiom of Choice fails. We force over this model to add a function from this Dedekind-finite set to some infinite ordinal $\kappa$. In the case that we force the function to be injective, it turns out that the resulting model is the same as adding $\kappa$ Cohen reals to the ground model, and that we have just added an enumeration of the canonical Dedekind-finite set. In the case where the function is merely surjective it turns out that we do not add any reals, sets of ordinals, or collapse any Dedekind-finite sets. This motivates the question if there is any combinatorial condition on a Dedekind-finite set $A$ which characterises when a forcing will preserve its Dedekind-finiteness or not add new sets of ordinals. We answer this question in the case of ``Adding a Cohen subset'' by presenting a varied list of conditions each equivalent to the preservation of Dedekind-finiteness. For example, $2^A$ is extremally disconnected, or $[A]^{<\omega}$ is Dedekind-finite.
\end{abstract}
\maketitle
\section{Introduction}
Cohen developed the method of forcing to prove that Cantor's Continuum Hypothesis is not provable from the axioms of Zermelo--Fraenkel and the Axiom of Choice. He then quickly adapted the known techniques for producing models where the Axiom of Choice fails using atoms (or urelements) to match the method of forcing. His models, therefore, proved that the Axiom of Choice does not follow from the rest of the axioms of Zermelo--Fraenkel set theory. In this model there is a canonical set of real numbers which is Dedekind-finite, that is infinite but without a countably infinite subset.

Over the years we see time and time again how rich and interesting the theory of Cohen's first model is. Recently, for example, Beriashvili, Schindler, Wu, and Yu proved in \cite{BSWY:2018} that in Cohen's first model there is a Hamel basis for the real numbers as a vector space over the rational numbers.

From a modern perspective, Cohen's first model is constructed by adding a countable sequence of Cohen reals, and then ``forgetting the enumeration, while remembering the set'' using a method called symmetric extensions that extends the method of forcing and is the main tool in proving consistency results related to the Axiom of Choice. We give an overview of this technique in \autoref{sec:prelim} and an overview of Cohen's first model in \autoref{sec:cohens-model}.

In this paper we show that forcing over Cohen's first model can have some counterintuitive results. Our two main results to that effect are \autoref{thm:injective-collapse}, which shows that we can introduce an arbitrary enumeration of the canonical Dedekind-finite set and the resulting model is itself the appropriate Cohen extension of the ground model; and \autoref{thm:collapse} where we show that an analogue of the Levy collapse adds a surjection from the canonical Dedekind-finite set onto any fixed ordinal, but does not add new sets of ordinals. In particular, this forcing preserves Dedekind-finiteness of sets in Cohen's first model.

Forcing (or generic) extensions of Cohen's first model were studied by Monro in \cite{Monro:1983}, where he shows that it is possible to add by a forcing extension a set which cannot be linearly ordered. This shows that, unlike the Axiom of Choice, the statement ``every set can be a linearly ordered'' is not preserved when taking generic extensions. This line of study was extended more recently by Hall, Keremedis, and Tachtsis in \cite{HKT:2013} where the authors study Monro's model, as well as a generic extensions of their own doing, in order to show the unprovability of certain weak choice principles related to ultrafilters on $\omega$.

Some of our arguments (\autoref{thm:collapse} in particular) are based on ideas in another work of Monro \cite{Monro:1973} (later developed by the first author in \cite{Karagila:2019} and by Shani in \cite{Shani:2018}), where adding Cohen subsets to Dedekind-finite sets is used iteratively to prove the independence of certain weak choice principles from one another. In these works it is crucial that no new sets of ordinals are added, and in particular no real numbers. When working over Cohen's first model, this means that Dedekind-finiteness is preserved by adding a Cohen subset.

In \autoref{sec:characterisations} we study Dedekind-finite sets which remain Dedekind-finite after adding a Cohen subset to them. We give 10 different equivalent conditions for this preservation, and we show that if the Dedekind-finite set is a set of real numbers, like in Cohen's first model, then these conditions are satisfied.
\section{Preliminaries}\label{sec:prelim}
Since we are dealing with models of $\ZF$ and with cardinals, it is worth clarifying what we mean by cardinals. We say that a set $X$ \textit{can be well-ordered}, or that it is \textit{well-orderable}, if there is an ordinal $\alpha$ and a bijection $f\colon X\to\alpha$. If $X$ can be well-ordered, the cardinal of $X$ is the smallest ordinal in bijection with $X$. If, however, $X$ cannot be well-ordered, we utilise Scott's trick and define its cardinal as the set $\{Y\in V_\alpha\mid\exists f\colon X\to Y\text{ a bijection}\}$, where $\alpha$ is the least ordinal for which the set is not empty. The letters $\kappa$ and $\lambda$ will always denote well-ordered cardinals.

We say that a set $X$ is \textit{Dedekind-finite} if it has no countably infinite subset, although we will use the term Dedekind-finite exclusively to mean that it is also infinite, as finite sets already have a much shorter name. It is a simple exercise to prove that $X$ is Dedekind-finite if and only if every injection $f\colon X\to X$ is a bijection. We note that if $X$ is a Dedekind-finite set which can be linearly ordered, e.g.\ a subset of the real numbers, then $[X]^{<\omega}=\{a\subseteq X\mid a\text{ is finite}\}$ is Dedekind-finite as well. This is because every finite subset of $X$ can be uniformly enumerated, so the union of countably many finite subsets will be a countable subset of $X$. Note that it is possible that $X$ is a Dedekind-finite set while $[X]^{<\omega}$ is Dedekind-infinite, for example $\bigcup\{P_n\mid n<\omega\}$ where each $P_n$ is a pair, and no infinite family of pairs admits a choice function (such set is sometimes referred to as a Russell set, or a socks set).

Our forcing terminology is standard. We say that $\PP$ is a notion of forcing if it is a partially ordered set with a maximum, $\1_\PP$. We call the elements of $\PP$ conditions, and we say that a condition $q$ is \textit{stronger} than a condition $p$, or that it \textit{extends} $p$, if $q\leq p$. We also follow Goldstern's alphabet convention which dictates that if $p,q$ are both conditions in $\PP$, then $p$ will not denote a stronger condition than $q$. If two conditions have a joint extension we say that they are \textit{compatible}, otherwise they are \textit{incompatible}. $\PP$-names are denoted by $\dot x$, and canonical names for ground model objects are denoted by $\check x$ when $x$ is the object in the ground model.

If $X$ is a set, $\Add(\omega,X)$ is the forcing whose conditions are finite partial functions $p\colon X\times\omega\to 2$. We denote by $\supp(p)$ the support of $p$, which is the maximal (finite) subset of $X$ such that $p\colon\supp(p)\times\omega\to2$. In the context of $\ZF$, if $A$ is a set, we denote by $\Add(A,X)$ the set of partial functions $p\colon X\times A\to 2$ such that $\dom p$ can be well-ordered and $|p|<|A|$.

Finally, given a family of $\PP$-names, $\{\dot x_i\mid i\in I\}$, we denote by $\{\dot x_i\mid i\in I\}^\bullet$ the obvious name this family defines, that is $\{\tup{\1_\PP,\dot x_i}\mid i\in I\}$. This notation extends to ordered pairs and to sequences as well. Using this notation, $\check x=\{\check y\mid y\in x\}^\bullet$.
\subsection{Symmetric extensions}\label{subsection:sym-ext}
The method of forcing, albeit very useful, preserves the Axiom of Choice when it holds in the ground model. In order to accommodate consistency proofs related to the failure of the Axiom of Choice we need to extend the method of forcing. Let $\PP$ be a notion of forcing, and let $\pi$ be an automorphism of $\PP$. Then $\pi$ acts on $\PP$-names via this recursive definition:\[\pi\dot x=\{\tup{\pi p,\pi\dot y}\mid\tup{p,\dot y}\in\dot x\}.\]

Let $\sG$ be a group of automorphisms of $\PP$. We say that $\sF$ is a \textit{filter of subgroups over $\sG$} if it is a filter on the lattice of subgroups, namely it is a nonempty family of subgroups of $\sG$ closed under finite intersections and supergroups. We say that $\sF$ is \textit{normal} if whenever $H\in\sF$ and $\pi\in\sG$, $\pi H\pi^{-1}\in\sF$ as well.

We call $\tup{\PP,\sG,\sF}$ a \textit{symmetric system} if $\PP$ is a notion of forcing, $\sG$ is a subgroup of $\aut(\PP)$, and $\sF$ is a normal filter of subgroups over $\sG$. In all cases it is enough to require that $\sF$ is a basis for a normal filter instead of a filter.

Let us fix a symmetric system for the rest of the section. We denote by $\sym_\sG(\dot x)$ the group $\{\pi\in\sG\mid\pi\dot x=\dot x\}$, also called the stabiliser of $\dot x$. We say that $\dot x$ is \textit{$\sF$-symmetric} if $\sym_\sG(\dot x)\in\sF$. When $\dot x$ is $\sF$-symmetric, and this condition holds hereditarily for the names in $\dot x$, we say that $\dot x$ is \textit{hereditarily $\sF$-symmetric}. The class $\HS_\sF$ denotes the class of all hereditarily $\sF$-symmetric names. When the symmetric system is clear from the context, and here it will always be clear from context, we omit the subscripts.

The forcing relation can be relativised to $\HS$, and we use $\forces^\HS$ to denote this relativisation.

\begin{lemma*}[The Symmetry Lemma]
Let $p\in\PP$ be a condition, $\pi\in\aut(\PP)$, and let $\dot x$ be a $\PP$-name, then $p\forces\varphi(\dot x)\iff\pi p\forces\varphi(\pi\dot x)$.
Moreover, if $\pi\in\sG$ and $\dot x\in\HS$, then we can replace $\forces$ by $\forces^\HS$.
\end{lemma*}
The first part of the proof appears as Lemma~14.37 in \cite{Jech:ST2003}, and the last sentence is an easy consequence of the fact that $\dot x\in\HS$ if and only if $\pi\dot x\in\HS$.

\begin{theorem*}
Let $G\subseteq\PP$ be a $V$-generic filter and let $M=\HS^G=\{\dot x^G\mid\dot x\in\HS\}$. Then $M$ is a transitive class model of $\ZF$ in $V[G]$ such that $V\subseteq M$.
\end{theorem*}
The model $M$ in the theorem is called a \textit{symmetric extension (of $V$)}. The theorem appears in \cite{Jech:ST2003} as Theorem~15.51.

\section{Cohen's first model}\label{sec:cohens-model}
Cohen's first model is the classical example of a model of set theory where the Axiom of Choice fails. This model was investigated by Halpern and Levy in \cite{HalpernLevy:1967} where they prove that every set in the model can be linearly ordered, and much more. (The model is sometimes referred to as the Halpern--Levy, or the Cohen--Halpern--Levy model.) This model has many presentations in the literature (see Chapter~5 in \cite{Jech:AC} for a comprehensive analysis of the construction, for example). For convenience of the reader we give a brief overview of the construction here as well.

We assume that $V$ satisfies $\ZFC$,\footnote{Traditionally $V$ is taken as $L$, but this is not important.} and we take $\PP$ to be $\Add(\omega,\omega)$. Our group of automorphisms is given by the group of finitary permutations of $\omega$ acting on $\PP$ in the natural way: \[\pi p(\pi n,m)=p(n,m),\] or equivalently: $\pi p(n,m)=p(\pi^{-1} n,m)$.\footnote{While the latter definition seems more natural from a naive point of view, it is in fact the former that is easier to work with.} And finally, $\sF$ is the filter of subgroups generated by $\{\fix(E)\mid E\in[\omega]^{<\omega}\}$, where $\fix(E)=\{\pi\in\sG\mid \pi\restriction E=\id\}$.\footnote{Because we work with pointwise stabilisers of finite sets, using the full symmetry group of $\omega$ is the same as using finitary permutations.} If $\fix(E)\subseteq\sym(\dot x)$, we say that $E$ is a \textit{support} for $\dot x$.

For each $n$, define the name $\dot a_n$ as the canonical name for the $n$th Cohen real, i.e.\ $\{\tup{p,\check m}\mid p(n,m)=1\}$, and let $\dot A=\{\dot a_n\mid n<\omega\}^\bullet$.

\begin{claim}
If $\pi\in\sG$, then $\pi\dot a_n=\dot a_{\pi n}$, therefore $\pi\dot A=\dot A$. Consequently, $\dot A\in\HS$.
\end{claim}
\begin{proposition}
$\1\forces^\HS\dot A$ is Dedekind-finite.
\end{proposition}
\begin{proof}
Suppose that $\dot f\in\HS$ and $p\forces^\HS\dot f\colon\check\omega\to\dot A$. Let $E$ be a support for $\dot f$, and without loss of generality $\supp(p)\subseteq E$ as well.

Let $n\notin E$ be some natural number, and assume towards contradiction that $q\leq p$ is a condition such that $q\forces^\HS\dot f(\check m)=\dot a_n$ for some $m<\omega$. Let $j<\omega$ be some natural number such that $j\notin E\cup\supp(q)$, and let $\pi$ be the $2$-cycle $(n\ j)$. Then the following hold:
\begin{enumerate}
\item $\pi\in\fix(E)$ and therefore $\pi p=p$ and $\pi\dot f=\dot f$.
\item $\pi\dot a_n=\dot a_j$.
\item $\pi q$ is compatible with $q$, since the only coordinates changed between $\pi q$ and $q$ are $j$ and $n$, but these are mutually exclusive to the conditions.
\item $\pi q\forces^\HS\pi\dot f(\pi\check m)=\pi\dot a_n$ which is to say, by the above, $\pi q\forces^\HS\dot f(\check m)=\dot a_j$.
\end{enumerate}
Therefore $q\cup\pi q\forces^\HS``\dot a_n=\dot f(\check m)=\dot a_j$ and $\dot a_n\neq\dot a_j"$. This is impossible, therefore the assumption that there are such $q$ and $m$ must be false. Therefore, $p$ forces that the range of $\dot f$ is finite, and in fact a subset of $\{\dot a_n\mid n\in E\}^\bullet$, so in particular, $p$ must force that $\dot f$ is not injective.
\end{proof}
In the following two sections we work in the Cohen model. We fix a $V$-generic filter $G\subseteq\PP$ and denote by $M$ the symmetric extension obtained by it and the symmetric system defined here. We will write $a_n$ and $A$ to denote $\dot a_n^G$ and $\dot A^G$ respectively.

\begin{remark}
One can prove that Cohen's model can be presented as $V(A)$, namely the smallest transitive model of $V[G]$ containing $V$ and having $A$ as an element, where $G\subseteq\PP$ is $V$-generic; or alternatively it is $\HOD_{V,A\cup\{A\}}^{V[G]}$, i.e.\ the class of all sets in $V[G]$ which are hereditarily definable from an element of $V$ and finitely many elements of $A$ and $A$ itself.

The idea that $V(A)$ is the symmetric extension is relatively straightforward, and it is worth sketching the argument behind it. On the one hand, since $V\subseteq M$, and $A\in M$ we immediately have $V(A)\subseteq M$. On the other hand, by analysing the proof of Lemma~5.25 and Lemma~5.26 in \cite{Jech:AC}, we see that if $x\in M$, then we can assign to it a minimal finite subset, $A_0$, of $A$ and a name in $V$ such that $A_0$ is the copy of a support of the name, and from this we can define $x$ using $A_0,A$ and the name from $V$ as parameters. By induction on $\rank x$ we get that $M\subseteq V(A)$.
\end{remark}
\section{Injective collapse}\label{sec:injective}
For two sets $X$ and $Y$, define $\iCol(X,Y)$ as the partial order given by well-orderable partial injections $p\colon X\to Y$. We note that $\iCol(X,Y)$ is isomorphic to $\iCol(Y,X)$. In this section we will focus on $\iCol(A,\kappa)$ when $\kappa$ is an infinite ordinal,\footnote{We can of course assume it is a cardinal, but the assumption is never used.} and since $A$ is Dedekind-finite, the conditions are finite. It will be easier to work with $\iCol(\kappa,A)$ instead, and as it is isomorphic to $\iCol(A,\kappa)$, we can do that without a problem.

Let $f\colon\kappa\to\omega$ be a finite partial injection, and let $\dot q_f$ denote the following name: \[\dot q_f=\left\{\tup{\check\alpha,\dot a_{f(\alpha)}}^\bullet\mid\alpha\in\dom f\right\}^\bullet.\]
\begin{claim}\label{claim:canonical-names}
If $\iCol(\check\kappa,\dot A)^\bullet=\{\dot q_f\mid f\colon\kappa\to\omega\text{ is a finite partial injection}\}^\bullet$, then $(\iCol(\check\kappa,\dot A)^\bullet)^G=\iCol(\kappa,A)$. In particular, if $p$ forces that $\dot q$ is a name for a condition in $\iCol(\kappa,A)$, then there is $p'\leq p$ and $f$ such that $p'\forces\dot q=\dot q_f$.
\end{claim}
\begin{theorem}\label{thm:injective-collapse}
Let $F\colon\kappa\to A$ be an $M$-generic function for $\iCol(\kappa,A)$. Then $F$ is $V$-generic for $\Add(\omega,\kappa)$. In particular, $M[F]=V[F]$, and all cardinals are preserved.
\end{theorem}
Since $F$ is not a filter, by $M$-generic we mean that in every dense $D\subseteq\iCol(\kappa,A)$ in $M$, there is some $p\in D$ such that $p\subseteq F$. In other words, $F$ is the function given by the generic filter.
\begin{proof}
For a pair $p\in\PP$ (where $\PP=\Add(\omega,\omega)$) and a name $\dot q_f$, let $r_{p,f}$ be the condition in $\Add(\omega,\kappa)$ defined by $r_{p,f}(\alpha,n)=p(f(\alpha),n)$.

%Let $\dot D$ be a name in $\HS$ such that some $p_0\forces^\HS\dot D$ is a dense open subset of $\iCol(\check\kappa,\dot A)^\bullet$. Let $D^*$ be $\{r_{p,f}\mid p\forces\dot q_f\in\dot D\}$. We claim that $D^*$ is a dense open subset of $\Add(\omega,\kappa)$. To see that, suppose that $r\in\Add(\omega,\kappa)$ is any condition, and let $f'\colon\supp(r)\to\omega$ be some injective function such that $\rng f'\cap\dom p=\varnothing$. Then $\dot q_{f'}$ is a name for a condition in $\iCol(\kappa,A)$. There is some $p\leq p_0$ such that $p(f'(\alpha),n)=r(\alpha,n)$ and some $f$ which extends $f'$, such that $p\forces\dot q_f\in\dot D$. Then $r_{p,f}\in D^*$, and if $\tup{\alpha,n}\in\dom r$, then $r(\alpha,n)=p(f'(\alpha),n)=r_{p,f}(\alpha,n)$, and therefore $r_{p,f}\leq r$.

If $D^*$ is a dense subset of $\Add(\omega,\kappa)$, then we define a name for a subset of $\iCol(\kappa,A)$:\[\dot D=\{\tup{p,\dot q_f}\mid\exists r\in D^*\text{ such that }f\colon\supp(r)\to\omega\text{ is injective and }r=r_{p,f}\}.\]
We claim that $\dot D$ is a name for a dense set. Suppose that $\dot q_{f'}$ is any condition. Let $p'$ be some condition such that $\supp(p')=\rng(f')$ (we may extend $f'$ if necessary, thus strengthening $\dot q_{f'}$), and let $r'=r_{p',f'}$. By density there is some $r\in D^*$ such that $r\leq r'$, then we can extend $f'$ to any injective $f$ and define $p$ by $p(f(\alpha),n)=r(\alpha,n)$. Then by definition, $\tup{p,\dot q_f}\in\dot D$ so $p\forces\dot q_f\in\dot D$. But $p\leq p'$ and $\dot q_{f'}\subseteq\dot q_f$. In other words, for every name $\dot q_{f'}$ for a condition in $\iCol(\kappa,A)$ we showed that every condition in $\PP$ can be strengthened to one which forces some extension of $\dot q_{f'}$ to be in $\dot D$.

Suppose now that $F$ is an $M$-generic function for $\iCol(\kappa,A)$, and let $D^*\in V$ be a dense open subset of $\Add(\omega,\kappa)$. Let $\dot D$ be the name obtained from $D^*$ as above, since $F$ is $M$-generic, there is some $p\in G$ and $\dot q_f$ such that for some $r\in D^*$, $r=r_{p,f}$ and $\dot q_f^G\subseteq F$. But this means that $r\subseteq F$, when $F$ is seen as a function from $\kappa\times\omega\to 2$, replacing each Cohen real in $A$ by its characteristic function. Therefore $F$ is $V$-generic for $\Add(\omega,\kappa)$ as wanted. Finally, since $A=\rng F$ and $M=V(A)\subseteq V[F]$, we have $M[F]=V[F]$ as well.
\end{proof}

\begin{corollary}
  The symmetric extension of $V$ given by imitating Cohen's first model, using $\Add(\omega,\kappa)$ (with finitary permutations of $\kappa$, etc.) and $F$ as in \autoref{thm:injective-collapse} as the generic, is $M$.
\end{corollary}
\begin{proof}
This is true since $M=V(A)$, and the argument for this equality is the same even when using $\Add(\omega,\kappa)$, which is easy to see from analysing the same proofs as in the case $\kappa=\omega$.\footnote{We can also obtain the same by applying Feferman's theorem appearing as Problem~23 in Chapter~5 of \cite{Jech:AC} instead of the proof analysis.}
\end{proof}

\begin{remark}
The corollary means that the process works in reverse as well, namely, starting with $\Add(\omega,\lambda)$ with finitary permutations of $\lambda$ and a filter of subgroups generated by pointwise stabilisers of finite subsets of $\lambda$, we end up with an analogue of the Cohen model where we have a set of Cohen reals which is Dedekind-finite. Using finite injective functions $f\colon\kappa\to\lambda$ provides us with the same proof as above.
\end{remark}

This result is odd. Indeed, upon first reading, it makes no sense. It quite literally implies that there is a bijection between $\omega$ and $\kappa$. But we should point out that all it implies \textit{inside $V$} is that there is a \textit{generic} bijection between them, i.e.\ we can generically collapse $\kappa$ to be countable. Moreover, the generic objects that we always refer to are not guaranteed to exist ``out of the blue'', rather we have a working assumption that $V$ is some countable transitive model in a larger universe. And of course, under this assumption, $\kappa$ is in fact a countable ordinal.

In the case where $\kappa$ is singular of countable cofinality (recall that we only required that $\kappa$ is infinite), it is well-known that adding $\aleph_\omega$ Cohen reals (to a model of $\CH$, at least), adds $\aleph_{\omega+1}$ of them. When we move from $A$ having order-type $\omega_1$ to $\omega_\omega$, we seemingly add a lot more reals, which will soon disappear as we move again, say to $\omega_2$. This is the place to point out, of course, that the additional reals are the consequence of being able to define new reals from certain infinite subsequences of the generic, none of which is symmetric enough to enter the Cohen model itself.

And finally, a question.

\begin{question}
It is known that the Cohen model is rather impoverished when it comes to variety of Dedekind-finite set. Indeed, every Dedekind-finite set can be taken as a subset of $\omega\times[A]^{<\omega}$. Will the results be the same if we replace $A$ by any other Dedekind-finite set in the Cohen model?
\end{question}
One should make the obvious, and immediate, observation that taking the above question at face value the answer is no. Split $A$ into two infinite parts, $A_0$ and $A_1$ (e.g., those reals which include $0$ and those that omit it), and force with $\iCol(\kappa,A_0)$ instead. It is not hard to see that we do not add any enumeration of $A_1$, which therefore remains Dedekind-finite. However over the symmetric model that is $V(A_0)$, the result was indeed the same as above.
\section{``Levy collapse'' without adding reals (or sets of ordinals)}\label{sec:levy-collapse}
For two sets $X$ and $Y$, denote by $\Col(X,Y)$ the set of partial functions $p\colon X\to Y$ such that $|p|$ is well-ordered and $|p|<|X|$, ordered by reverse inclusion. This coincides with the standard definition when $X$ can be well-ordered, but we will focus on the case where $X=A$ in Cohen's first model, meaning that the conditions are, as before, finite functions.

In a manner similar to \autoref{claim:canonical-names}, if $q\in\Col(A,\kappa)$ is a condition, where $\kappa$ is some well-ordered cardinal, then there is some $f\colon\omega\to\kappa$ in $V$ such that $q=\dot q_f^G$, with $\dot q_f$ defined as before.
\begin{theorem}\label{thm:collapse}
Let $\kappa$ be an infinite cardinal and let $G\subseteq\Col(A,\kappa)$ be an $M$-generic filter. Then $M$ and $M[G]$ have the same sets of ordinals.
\end{theorem}
\begin{proof}
Let $\dot X\in M$ be a $\Col(A,\kappa)$-name for a set of ordinals. It is easier to consider $\dot X$ as a name in the iteration $\PP\ast\Col(\dot A,\check\kappa)^\bullet$, whose projection to a $\PP$-name of a $\Col(A,\kappa)$-name, denoted by $[\dot X]$, is in $\HS$.\footnote{While this is quite simple to understand in this case, a more general theory of iterations of symmetric extensions was developed by the first author in \cite{Karagila:2019}.} Moreover, since we have such canonical names for conditions in $\Col(A,\kappa)$, and we are only interested in this iteration of two steps, we may assume that the conditions in this iteration have the form $\tup{p,\dot q_f}$.

Note that if $\pi\in\sG$, then $\pi$ acts on $\PP\ast\Col(\dot A,\check\kappa)^\bullet$ in the obvious way: \[\pi\tup{p,\dot q_f}=\tup{\pi p,\pi\dot q_f}=\tup{\pi p,\dot q_{f\circ\pi}}.\]
Let $\tup{p,\dot q_f}$ be a condition which forces that $\dot X$ is a name for a set of ordinals, and let $E$ be a support for $\dot X$, i.e.\ a finite subset of $\omega$ for which $\fix(E)\subseteq\sym([\dot X])$. We may assume that $\supp(p)=E=\dom f$. Suppose that $\tup{p_0,\dot q_{f_0}}$ and $\tup{p_1,\dot q_{f_1}}$ are two extensions of $\tup{p,\dot q_f}$. Again, we may assume that $\supp(p_i)=\dom f_i$ for $i<2$.

If $p_1\restriction E=p_2\restriction E$, then the two must agree on any statement of the form $\check\alpha\in\dot X$. This is because there is an automorphism in $\fix(E)$ moving $\supp(p_0)\setminus E$ to be disjoint of $\supp(p_1)$, which means that $\tup{\pi p_0,\pi\dot q_{f_0}}$ is compatible with $\tup{p_1,\dot q_{f_1}}$ while $\pi\check\alpha=\check\alpha$ and $\pi\dot X=\dot X$. Here we used the fact that $\dom f=E$, as well as the fact the conditions are not injective. Indeed, for injective functions agreeing on their common domain is not enough to be compatible.

In particular this means that if $\tup{p',\dot q_{f'}}\leq\tup{p,\dot q_f}$ and $\tup{p',\dot q_{f'}}\forces\check\alpha\in\dot X$, then $\tup{p'\restriction E,\dot q_{f'\restriction E}}=\tup{p'\restriction E,\dot q_f}$ already forced this statement, and the same for $\check\alpha\notin\dot X$, of course. Therefore the conclusion follows, and therefore $\dot X$ is a name for a set in $M$, given by the name $\dot X_f=\{\tup{p',\check\alpha}\mid\tup{p'\restriction E,\dot q_f}\forces\check\alpha\in\dot X\}$.
\end{proof}
This provides another proof of the known fact (see Problem~16 in Chapter~5 of \cite{Jech:AC}) that two models of $\ZF$ with the same sets of ordinals are not necessarily equal.\footnote{This requires, of course, that the Axiom of Choice fails in both models.}
\begin{corollary}
Forcing with $\Col(A,\kappa)$ over $M$ preserves cofinalities.\qed
\end{corollary}
\begin{corollary} \label{cor:Adding a Cohen subset to Cohen's Dedekind finite set}
$A$ is still Dedekind-finite after forcing with $\Col(A,\kappa)$.
\end{corollary}
\begin{proof}
Suppose not, then there is an injective function $f\colon\omega\to A$ in $M[G]$, where $G$ is $M$-generic for $\Col(A,\kappa)$, and $f$ can be coded as a real, namely a subset of $\omega$. However, by \autoref{thm:collapse} no new reals are added, and therefore $f\in M$. This is a contradiction since $A$ is Dedekind-finite in $M$.
\end{proof}
\begin{corollary}
Every Dedekind-finite set remains Dedekind-finite after forcing with $\Col(A,\kappa)$.
\end{corollary}
\begin{proof}
There is an injection from every set in $M$ into $[A]^{<\omega}\times\eta$ for some ordinal $\eta$, therefore for a Dedekind-finite set in $M$ we can take $\eta=\omega$. But this means that if $A$ remains Dedekind-finite, so must $[A]^{<\omega}$, as $A$ is a set of real numbers, and therefore every other Dedekind-finite set remain Dedekind-finite as well.
\end{proof}
In $M$, the set $A$ has a partition into $\aleph_0$ different parts (e.g., by looking at $\min a$ for $a\in A$, which by genericity must obtain each possible value infinitely many times). After forcing with $\Col(A,\kappa)$ we added new partitions of size $\kappa$, without adding new sets of ordinals. This is in contrast to the results of Monro in \cite{Monro:1983}: in his model the generic partition is an infinite partition of $A$ which itself cannot be split into two infinite sets making it.
\begin{question}
In the previous section and in this one, the proofs involved in a fairly meaningful way the Cohen forcing itself. What happens when we consider a similar symmetric extension obtained by using a different kind of real numbers, e.g.\ random reals, Sacks reals, etc., or even a mixture of these? On its face it seems that the proof uses more of the fact that we take a finite-support product of infinitely many copies of the same forcing, rather than the specific properties of the Cohen forcing. To what extent can this be pushed?
\end{question}

\section{Adding a Cohen subset to a Dedekind-finite set}\label{sec:characterisations}

\autoref{cor:Adding a Cohen subset to Cohen's Dedekind finite set} shows that the Dedekind-finite set $A$ in Cohen's first model is still Dedekind-finite after forcing with $\Col(A,\kappa)$, and this leads to the problem of finding more general condition that ensure this. In this section, we provide characterisations, in $\ZF$, of those Dedekind-finite sets $A$ that remain Dedekind-finite after forcing with $\Add(A,1)=\Col(A,2)$. The combined results can already be stated in the following result, although some of the notions in the theorem are only defined below.

\begin{theorem}\label{thm:characterisations}
Let $A$ be a Dedekind-finite set. The following are equivalent:
\begin{enumerate-(1)}
\item \label{c:finite-subsets-is-DF}
$[A]^{<\omega}$ is Dedekind-finite.
\item \label{c:Add-has-finite-cc}
$\Add(A,1)$ contains no infinite antichains.
\item \label{c:Add-has-omega-cc}
$\Add(A,1)$ contains no countably infinite antichains.
\item \label{c:Add-has-finite-decision-p}
$\Add(A,1)$ has the finite decision property.
\item \label{c:Add-preserves-DF-of-A}
$A$ remains Dedekind-finite in any generic extension by $\Add(A,1)$.
\item \label{c:Add-does-not-collapse-A}
$A$ is not collapsed in any generic extension by $\Add(A,1)$.
\item \label{c:Add-does-not-add-Cohens}
$\Add(A,1)$ fails to add a Cohen real.
\item \label{c:Add-does-not-add-reals}
$\Add(A,1)$ fails to add a real.
\item \label{c:Add-does-not-add-sets-of-ord}
$\Add(A,1)$ fails to add a set of ordinals.
\item \label{c:Cantor-cube-is-ext-disc}
$2^A$ is extremally disconnected.
\end{enumerate-(1)}
\end{theorem}
Some of these conditions (e.g., \ref{c:finite-subsets-is-DF}) already imply that $A$ is Dedekind-finite, but others do not (e.g., \ref{c:Add-does-not-add-Cohens} holds for $A=\omega_1$ when assuming $\ZFC$, and \ref{c:Add-does-not-collapse-A} holds for $A=\omega$ even in $\ZF$).

The forcing $\Add(A,1)$ was recently studied by Goldstern and Klausner in \cite{GoldsternKlausner} where the authors study the possible effects of this forcing on the structure of cardinals, as well as some specific properties of the forcing, e.g.\ condition \ref{c:Add-has-finite-cc} in our theorem, that occur when $A$ is assumed to be Dedekind-infinite with particular properties.

Our theorem admits an easy corollary, which is applicable to Cohen's first model.
\begin{corollary}
If $A$ is a Dedekind-finite set which can be linearly ordered, in particular a set of real numbers, then all the conditions of \autoref{thm:characterisations} hold.
\end{corollary}

In the proofs, we will use the \textit{sunflower lemma}, a finite version of the $\Delta$-system lemma. Before we state the lemma, we fix the following notation. A \textit{sunflower} is a collection of sets, $S$, such that for some $t$, $u\cap v=t$ for all $u\neq v$ in $S$. Moreover, the set $t$ is called the \textit{centre} of the sunflower.

\begin{lemma}[Erd\H{o}s--Rado \cite{ErdosRado:1960}] \label{lemma:sunflower}
  If $a$ and $b$ are positive integers, then any collection of $b!ab+1$ sets of size ${\leq}b$ contains a sunflower of size ${>}a$.
\end{lemma}
The above lemma, involving only finite sets, is of course provable without the Axiom of Choice. Less obvious, though, is that the following lemma can also be proved in $\ZF$, as was done by Keremedis, Howard, Rubin, and Rubin in \cite{HKRR:1998}.\footnote{We can actually use this lemma instead of the Erd\H{o}s--Rado lemma in the arguments in this paper.}
\begin{lemma}[Lemma~4 in \cite{HKRR:1998}]\label{lemma:HKRR}
If $X$ is an infinite collection of sets of size $b$, for a natural number $b>1$, then there is a finite $t\subseteq\bigcup X$ such that for every positive number $a$, there is a subset of $X$ of size $a$ which is a sunflower with centre $t$.
\end{lemma}

It is easy to see that $[A]^{<\omega}$ is Dedekind-infinite if and only if there is a sequence $\vec{A}=\tup{A_n\mid n<\omega}$ of disjoint nonempty finite subsets of $A$. We will call such a sequence a \textit{disjoint sequence}.
We fix some more notation: for any $K\subseteq\Add(A,1)$ we define $\supp^K=\{\dom p\mid p\in K\}$. Note that $K$ is finite if and only if $\supp^K$ is finite.

\begin{lemma} \label{lemma:characters-of-finite-cc}
The following are equivalent:
\begin{enumerate-(a)}
\item \label{fcc:countable-antichain}
$\Add(A,1)$ contains a countably infinite antichain.
\item \label{fcc:infinite-antichain}
$\Add(A,1)$ contains an infinite antichain.
\item \label{fcc:finite-sets-is-DF}
$[A]^{<\omega}$ is Dedekind-infinite.
\end{enumerate-(a)}
\end{lemma}
\begin{proof}
\ref{fcc:countable-antichain} $\implies$ \ref{fcc:infinite-antichain} is clear.

\ref{fcc:infinite-antichain} $\implies$ \ref{fcc:finite-sets-is-DF}: Suppose that $\Add(A,1)$ contains an infinite antichain $C$. We show that $C_k=\{p\in C\mid |\dom p|=k\}$ is finite for all $k<\omega$. It then follows that $D=\bigcup_{k<\omega} \supp^{C_k}$ is an infinite union of finite sets and hence $[A]^{<\omega}$ is Dedekind-infinite.
Towards a contradiction, suppose that $C_k$ is infinite for some $k<\omega$. Then $\supp^{C_k}$ contains arbitrarily large sunflowers by \autoref{lemma:sunflower}. Since their centres are all of size at most $k$, we can find two compatible conditions in $C_k$, contradicting the assumption that $C$ is an antichain.

\ref{fcc:finite-sets-is-DF} $\implies$ \ref{fcc:countable-antichain}: If $[A]^{<\omega}$ is Dedekind-infinite, let $\vec{A}=\tup{A_n\mid n<\omega}$ be a disjoint sequence and $B_n=\bigcup_{i\leq n} A_i$. Define $p_0\colon B_0\to\{0,1\}$ to take the value $1$ on $B_0$, and for $n\in\omega\setminus\{0\}$, define $p_n\colon B_n\to \{0,1\}$ to take value $1$ on $A_n$ and $0$ on $B_{n-1}$. Then $C=\{p_n\mid n<\omega\}$ is a countably infinite antichain.
\end{proof}

The equivalence of \ref{fcc:countable-antichain} and \ref{fcc:finite-sets-is-DF} was independently proved recently by Keremedis and Tachtsis, see Lemma~1 in \cite{KeremedisTachtsis:2020}. They prove in Theorem~3 a further equivalence in terms of the topological \textit{cellularity} of the space $2^A$.

\begin{definition}
We say that $\Add(A,1)$ has the \textit{finite decision property} if for all formulas $\varphi(\dot x)$, the set $M^{\varphi(\dot x)}$ of minimal elements of $N^{\varphi(\dot x)}=\{p\mid p\forces\varphi(\dot x)\}$ with respect to restriction is finite.
\end{definition}

\begin{lemma}\label{lemma:finite decision property}
The following are equivalent.\footnote{One can easily formulate this lemma and its proof without using forcing. Then $N^{\varphi(\dot x)}$ is replaced by any subset $N$ of $\Add(A,1)$ with the following property: if $M$ is a subset of $N$ that is dense below some $p\in N$, then $p\in M$. }
\begin{enumerate-(a)}
\item \label{fdp:Add-has-fdp}
$\Add(A,1)$ has the finite decision property.
\item \label{fdp:finite-subsets-is-DF}
$[A]^{<\omega}$ is Dedekind-finite.
\end{enumerate-(a)}
\end{lemma}
\begin{proof}
\ref{fdp:Add-has-fdp} $\implies$ \ref{fdp:finite-subsets-is-DF}: Suppose that $[A]^{<\omega}$ is Dedekind-infinite, and let $\tup{A_n\mid n<\omega}$ be a disjoint sequence witnessing that. Let $\dot x$ be the $\Add(A,1)$-name for the least $n$ such that the canonical generic subset of $A$ contains $A_n$. Let $\varphi(\dot x)$ denote a formula stating that $\dot x$ is even. It is then easy to see that $M^{\varphi(\dot x)}$ is infinite.

\ref{fdp:finite-subsets-is-DF} $\implies$ \ref{fdp:Add-has-fdp}: Suppose that $\varphi(\dot x)$ is a formula such that $M^{\varphi(\dot x)}$ is infinite. We can assume that $M^{\varphi(\dot x)}_k=\{p\in M^{\varphi(\dot x)}\mid |\dom p|=k\}$ is infinite for some $k<\omega$, since one can otherwise construct a disjoint sequence. Let $M=M^{\varphi(\dot x)}_k$. Since $M$ is infinite, $\supp^M$ is infinite as well. By \autoref{lemma:HKRR}, there is some finite $t\subseteq\bigcup\supp^M$ which is the centre of arbitrarily large sunflowers in $\supp^M$, fix such $t$.

There are only finitely many possible values that a condition can take on $t$, in fact at most $3^k$, where the third value stands for `undefined'. Therefore, there is a condition $q$ with the property: there are arbitrarily large subsets $K$ of $M$ such that $\supp^K$ forms a sunflower with centre $t$ and $p\restriction t=q$ for all $p\in K$. Since $q$ is a proper subset of $p$ if $|K|>1$, it must be that $q\nforces\varphi(\dot x)$, by the minimality of $p$.

\begin{subclaim}
$q\forces\varphi(\dot x)$.
\end{subclaim}
\begin{proof}
Otherwise take some $r\leq q$ with $r\forces\lnot\varphi(\dot x)$. Then $r$ is incompatible with all elements of $M$. Moreover, take a subset $K$ of $M$ as above with $|K|>|r|$. Since $\supp^K$ forms a sunflower with centre $t$ and $|K|>|r|$, there is some $s\in K$ with $\dom r\cap\dom s=t$. We further have $s\restriction t=q$ by the choice of $K$.  Therefore $r$ and $s$ are compatible. But this contradicts the fact that $r$ and $s$ force opposite truth values of $\varphi(\dot x)$.
\end{proof}

But this is a contradiction, as $q\nforces\varphi(\dot x)$. Therefore $M_k^{\varphi(\dot x)}$ is finite for all $k$, and we can construct a disjoint subset as wanted.
\end{proof}

\begin{definition}
A set $A$ is \textit{collapsed} in an outer model $W$ of $V$ if there is a subset of $A$, $B\in V$, such that $V\models|B|<|A|$, but $W\models|B|=|A|$.
\end{definition}

\begin{lemma}\label{lemma:DF-preservation}
Let $A$ be a Dedekind-finite set. The following are equivalent.
\begin{enumerate-(a)}
\item \label{df:finite-subsets-is-DF}
$[A]^{<\omega}$ is Dedekind-infinite.
\item \label{df:Add-makes-A-DInf}
$A$ is Dedekind-infinite in any generic extension by $\Add(A,1)$.
\item \label{df:Add-collapses-A}
$A$ is collapsed in any generic extension by $\Add(A,1)$.
\end{enumerate-(a)}
\end{lemma}
\begin{proof}
\ref{df:finite-subsets-is-DF} $\implies$ \ref{df:Add-makes-A-DInf}: Assume that $[A]^{<\omega}$ is Dedekind-infinite, and let $\tup{A_n\mid n<\omega}$ be a disjoint sequence witnessing that. Suppose that $G\subseteq\Add(A,1)$ is $V$-generic subset of $A$, and let $B$ be $\{a\mid\exists p\in G, p(a)=1\}$. Define $f(n)=a$ whenever $B\cap A_n=\{a\}$, which by a density argument happens infinitely often. This defines an injection from an infinite subset of $\omega$ into $A$, and therefore $A$ is Dedekind-infinite.

\ref{df:Add-makes-A-DInf} $\implies$ \ref{df:Add-collapses-A}: Suppose that $A$ is Dedekind-infinite in the generic extension, and let $f\colon A\to A$ be an injection which is not a bijection, then there is some $a\in A$ such that $f(x)\neq a$ for all $x\in A$. In $V$, by Dedekind-finiteness of $A$ there, the set $A\setminus\{a\}$ is strictly smaller in size, and thus witnessing that $A$ was collapsed.

\ref{df:Add-collapses-A} $\implies$ \ref{df:Add-makes-A-DInf}: Trivial.

\ref{df:Add-makes-A-DInf} $\implies$ \ref{df:finite-subsets-is-DF}: Assume towards a contradiction that $[A]^{<\omega}$ is Dedekind-finite, but suppose that $p\forces\dot f\colon\check\omega\to\check A$ is injective, and for simplicity, assume that $p=\1$. Let $\varphi(\check n,\check a)$ denote the formula $\dot{f}(\check n)=\check a$. Let $s^n_a$ be the set $\supp^{M^{\varphi(\check n,\check a)}}$ for $n<\omega$ and $a\in A$. By the finite decision property in \autoref{lemma:finite decision property}, $M^{\varphi(\check n,\check a)}$ and thus $s^n_a$ are finite.

\begin{subclaim}
For any $n<\omega$, $\{a\in A\mid s_a^n\neq\varnothing\}$ is finite.
\end{subclaim}
\begin{proof}
Fix $n<\omega$ and let $B$ denote $\{a\in A\mid s_a^n\neq\varnothing\}$. Towards a contradiction, suppose that $B$ is infinite. First assume that for all $k<\omega$, there are only finitely $a\in B$ with $|s_a^n|=k$. Then $A_k$ defined as $\bigcup_{|s_a^n|=k} s^n_a$ is finite as well. Since the sets $M^{\varphi(\check n,\check a)}$ are disjoint as $a\in A$ varies, $\vec{A}=\tup{A_k\mid k<\omega}$ has an injective infinite subsequence and hence $[A]^{<\omega}$ is Dedekind-infinite.

Now assume that for some $k<\omega$, there are infinitely many $a\in B$ with $|s_a^n|=k$. Note that $k\neq 0$ by the definition of $B$. Since the sets $M^{\varphi(\check n,\check a)}$ are disjoint as $a\in A$ varies, the set $S=\{s_a^n\mid a\in B\}$ is infinite. This set contains arbitrarily large sunflowers by \autoref{lemma:sunflower}.

Let $T$ be a sunflower in $S$ of size $3^k+1$. Since the centre of $T$ has size ${\leq}k$, there are at most $3^k$ possible values for restrictions to the centre of $T$. Hence we can find $a\neq b$ in $B$ and conditions $p\in M^{\varphi(\check n,\check a)}$ and $q\in M^{\varphi(\check n,\check a)}$ with $p$ compatible with $q$. But this contradicts the fact that conditions in $M^{\varphi(\check n,\check a)}$ and $M^{\varphi(\check n,\check b)}$ are pairwise incompatible.
\end{proof}

Note that $s^n_a=\varnothing$ implies that $M^n_a$ is either empty or contains only $\1$. Thus the subclaim implies that $M^n$ defined as $\bigcup_{a\in A} M^{\varphi(\check n,\check a)}$ is finite for all $n<\omega$. In other words, there are only finitely many possible choices for $\dot{f}(n)$. However,  $\1$ forces that $\dot{f}$ has infinite range and thus $\bigcup_{n<\omega}M^n$ is an infinite subset of $A$. This allows us to construct a disjoint sequence in $[A]^{<\omega}$, in contradiction to the assumption that $[A]^{<\omega}$ is Dedekind-finite.
\end{proof}
Note that by the homogeneity of $\Add(A,1)$ the above proof does not depend on the choice of a generic filter, and since collapsing $A$ or making it Dedekind-infinite can be stated as a formula whose free variables are all canonical ground model names, there is no dependence on any specific condition either.

\begin{lemma}\label{lemma:adding Cohen reals}
Let $A$ be a Dedekind-finite set. The following are equivalent.
\begin{enumerate-(a)}
\item \label{cohen:finite-sets-is-DF}
$[A]^{<\omega}$ is Dedekind-infinite.
\item \label{cohen:Add-adds-a-Cohen}
$\Add(A,1)$ adds a Cohen real.
\item \label{cohen:Add-adds-a-real}
$\Add(A,1)$ adds a real.
\item \label{cohen:Add-adds-a-set-of-ord}
$\Add(A,1)$ adds a set of ordinals.
\end{enumerate-(a)}
\end{lemma}
\begin{proof}
\ref{cohen:finite-sets-is-DF} $\implies$ \ref{cohen:Add-adds-a-Cohen}: Let $\vec{A}=\tup{A_n\mid n<\omega}$ be a disjoint sequence witnessing that $[A]^{<\omega}$ is Dedekind-infinite. Let $\dot x=\{\tup{p,\check n}\mid p[A_n]=\{0\}\}$. A density argument shows that $\dot x$ is a name for a Cohen real.

\ref{cohen:Add-adds-a-Cohen} $\implies$ \ref{cohen:Add-adds-a-real} $\implies$ \ref{cohen:Add-adds-a-set-of-ord} is trivial.

\ref{cohen:Add-adds-a-set-of-ord} $\implies$ \ref{cohen:finite-sets-is-DF}: Suppose that $\1$ forces that $\dot{X}$ is a new subset of some ordinal $\eta$, and let $\varphi(\check\alpha)$ denote the formula $\check\alpha\in\dot X$. By the finite decision property in \autoref{lemma:finite decision property}, $M^{\varphi(\check\alpha)}$ is finite for all $\alpha<\eta$. However, the union of the domains of conditions in $\bigcup_{\alpha<\eta}M^{\varphi(\check\alpha)}$ is infinite, since $\dot X$ is a name for a new set of ordinals. Thus it is easy to construct a disjoint sequence in $[A]^{<\omega}$.
\end{proof}

We equip $2^A$ with the product topology. Moreover, let $N_p=\{x\in 2^A\mid p\subseteq x\}$ denote the basic open set associated to $p\in\Add(A,1)$. Note that $2^A$ is a Hausdorff space. We will consider the following notion: a topological space is called \textit{extremally disconnected} if the closure of every open set is open. Note that for Hausdorff spaces, this strengthens the property of being \textit{totally disconnected}.

\begin{lemma} \label{lemma:extremally disconnected}
The following are equivalent.
\begin{enumerate-(a)}
\item \label{ext:Cantor-cube-ext-disc}
$2^A$ is extremally disconnected.
\item \label{ext:finite-sets-is-DF}
$[A]^{<\omega}$ is Dedekind-finite.
\end{enumerate-(a)}
\end{lemma}
\begin{proof}
\ref{ext:Cantor-cube-ext-disc} $\implies$ \ref{ext:finite-sets-is-DF}: Suppose that $[A]^{<\omega}$ is Dedekind-infinite. We will show that $2^A$ is not extremally disconnected. To see this, let $\vec{A}=\tup{A_n\mid n<\omega}$ be a disjoint sequence and $B_n=\bigcup_{i\leq n} A_i$. Define $p_0\colon B_0\to\{0,1\}$ to take the value $1$ on $B_0$, and for $n\in\omega\setminus\{0\}$, define $p_n\colon B_n\to \{0,1\}$ to take value $1$ on $A_n$ and $0$ on $B_{n-1}$. Then $U$ defined as the union $\bigcup_{n<\omega} N_{p_n}$ is open. But the closure of $U$ is $U\cup\{\mathbf 0\}$ which is not open, since it does not contain any open neighbourhood of $\mathbf 0$, where $\mathbf 0$ denotes the constant function $\mathbf 0(a)=0$ for all $a\in A$.

\ref{ext:finite-sets-is-DF} $\implies$ \ref{ext:Cantor-cube-ext-disc}:\footnote{This argument, like \autoref{lemma:finite decision property}, can be made without forcing. Here we can us the topology.} Suppose that $[A]^{<\omega}$ is Dedekind-finite. Let $N_I$ be an open set of $2^A$ given by the union $\bigcup_{p\in I}N_p$. We will show that its closure is open.

Note that $f\in 2^A$ fails to be in the closure of $N_I$ if and only if there is some $p\in\Add(A,1)$ such that $p\subseteq f$ and $N_p\cap N_I=\varnothing$, which equivalently means that $p\forces\dot g\notin\dot N_I$, where $\dot g$ is the canonical generic function $A\to\{0,1\}$ and $\dot N_I$ is the re-interpretation of the union $\bigcup_{p\in I}N_p$ in the generic extension. But by the finite decision property, which holds by \autoref{lemma:finite decision property}, there is a finite set, $M$, of minimal elements $p$ such that $p\forces\dot g\notin\dot N_I$. In particular, $C=\bigcup_{p\in M}N_p$ is a finite union of clopen sets and thus closed.
So the closure of $N_I$ is open, as wanted.
\end{proof}

We finish with a question of interest, as one should.
\begin{question}
As we observed at the beginning of this section, $\Add(A,1)$ is the same as $\Col(A,2)$.
Moreover, we proved in the previous section that $\Col(A,\kappa)$ does not make $A$ Dedekind-infinite, where $A$ is the canonical Dedekind-finite set in Cohen's first model. Is there a similar characterisation for $\Col(A,\kappa)$, or even more generally for an arbitrary forcing that adds a fresh subset to a Dedekind-finite set?
\end{question}
\section*{Acknowledgements}
The authors would like to thank Matti Rubin for inviting the second author to Israel in March~2015 where this project began, and to the Hausdorff Center for Mathematics in Bonn and to the Ben-Gurion University of the Negev for supporting him during this visit. Additional thanks go to the organisers of the HIF programme for extending an invitation to visit Cambridge in October 2015, where work on this project continued. We would like to express our gratitude to the reviewers whose suggestions helped improve not only the presentation of this paper, but also its connections to past and current research. Final acknowledgement goes to Mike Oliver who gave a talk titled ``How to have more things by forgetting where you put them'', which inspired the title of this work.

\subsection*{Funding statement}
The first author was supported by the Royal Society Newton International Fellowship, grant no.~NF170989. This project has received funding from the European Union's Horizon 2020 research and innovation programme under the Marie Sk\l odowska-Curie grant agreement No 794020 (IMIC) of the second author.
\bibliographystyle{amsplain}
\providecommand{\bysame}{\leavevmode\hbox to3em{\hrulefill}\thinspace}
\providecommand{\MR}{\relax\ifhmode\unskip\space\fi MR }
\providecommand{\MRhref}[2]{%
  \href{http://www.ams.org/mathscinet-getitem?mr=#1}{#2}
}
\providecommand{\href}[2]{#2}

\end{document}